\documentclass[11pt]{article}
\usepackage[toc,page]{appendix}
\usepackage{stmaryrd}
\usepackage{etoolbox}
\usepackage{pst-node}
\usepackage{tikz-cd} 
\appto\appendix{\addtocontents{toc}{\protect\setcounter{tocdepth}{0}}}

\usepackage{enumitem}
\usepackage{amsfonts,amsmath, amssymb,latexsym}
\usepackage{mathtools}
\usepackage{multirow,comment}
\def\Z{{\mathbb Z}}
\def\A{{\mathbb A}}

\def\SL{{\rm SL}}

\def\GL{{\rm GL}}

\def\Adj{{\rm Adj}}

\def\Sym{{\rm Sym}}

\def\O{{\mathcal O}}
\newcommand{\cO}{{\mathcal O}}

\def\P{{\mathbb P}}

\def\Aut{{\rm Aut}}
\def\Pic{{\rm Pic}}

\def\F{{\mathbb F}}

\renewcommand{\L}{\mathcal{L}}

\def\FF{{\mathcal F}}

\def\Q{{\mathbb Q}}

\def\H{{\mathcal H}}

\def\Z{{\mathbb Z}}
\def\G{{\mathbb G}}
\def\P{{\mathbb P}}
\def\F{{\mathbb F}}
\def\Q{{\mathbb Q}}
\def\C{{\mathbb C}}
\def\H{{\mathcal H}}

\def\Sym{{\text{Sym}}}

\def\fz1{{F_{\Z,1}}}

\def\Proj{{\rm Proj}}

\def\I{{\mathcal{I}}}
\DeclareMathOperator {\spec} {\mathrm{Spec}\,}
\def\pp{{\mathfrak{p}}}

\newtheorem{theorem}{Theorem}

\newtheorem{corollary}[theorem]{Corollary}

\newtheorem{lemma}[theorem]{Lemma}

\newtheorem{remark}[theorem]{Remark}

\newtheorem{proposition}[theorem]{Proposition}

\newenvironment{proof}{\noindent {\bf Proof:}}{$\Box$ \vspace{2 ex}}

\usepackage{xcolor,color,url,lmodern}

\usepackage{xcolor}
\usepackage{environ}
\definecolor{darkpurple}{rgb}{0.5, 0.0, 0.5}
\newif\ifshowcontent
\showcontentfalse
 
\usepackage{xcolor}
\usepackage{environ}
\definecolor{darkpurple}{rgb}{0.5, 0.0, 0.5}
\newif\ifshowcontent
\showcontenttrue

\NewEnviron{highlighted}{
    \ifshowcontent
    \noindent
    {\bfseries\color{darkpurple} \textit{Note:} \BODY}
    \par
  \fi
}

\usepackage{xcolor}
\usepackage{soul}
\definecolor{lightpurple}{rgb}{0.85, 0.75, 0.95}  
\sethlcolor{lightpurple}

\title{Finiteness theorems for some representations of $\GL_3$}
\author{Fatemehzahra Janbazi and Arul Shankar}

\date{}

\begin{document}
\maketitle

\begin{abstract}
Let $n\geq 2$ be an integer and let $K$ be a number field with ring of integers $\cO_K$. We prove that the set of ternary $n$-ic forms forms with coefficients in $\cO_K$ and fixed nonzero discriminant, breaks up into finitely many $\GL_3(\cO_K)$-orbits. This generalizes the result of Birch--Merriman \cite{BM} in the binary forms case.
We also prove a similar finiteness result on the $\GL_3(\cO_K)$-orbits of the $27$-dimensional representation of $\GL_3$ with highest weight $(4,2)$.
\end{abstract}

\section{Introduction}

A classical result of Birch and Merriman \cite{BM} states that, for fixed $n$, the set of integral binary $n$-ic forms with a fixed nonzero discriminant breaks up into finitely many $\GL_2(\Z)$-orbits. They also prove the analogous result over arbitrary number fields. 
In this article, we consider the space of ternary $n$-ic forms.
For $n\geq 2$, and a ring $R$, let $V_n(R)$ denote the set of ternary $n$-ic forms (i.e., homogeneous degree-$n$ polynomials in three variables) with coefficients in $R$. The group $\GL_3(R)$ acts on $V_n(R)$ via linear change of variables: $\gamma\cdot f(x,y,z)=f((x,y,z)\cdot \gamma)$. When $n=2$, this is the (prehomogeneous) representation of $\GL_3(R)$ on ternary quadratic forms. When $n=3$, the representation of $\GL_3(R)$ acting on ternary cubic forms is no longer prehomogeneous, but is coregular - the ring of relative polynomial invariants is freely generated by two elements. As $n$ grows, the invariant theory becomes very complicated, and explicit descriptions of the rings of invariants, much beyond Hilbert's finite generation theorem, are unknown. 

However, for all $n$, there is a classically known relative  polynomial invariant for the action of $\GL_3(R)$ on $V_n(R)$, namely, the discrimiant. We denote the discriminant by 
$\Delta_n\in\Z[V_n]$; it is a degree-$3(n-1)^2$ polynomial with integer coefficients. Then we prove the following result.

\begin{theorem}\label{thm1}
Let $n\geq 2$ be a fixed integer. Let $K/\Q$ be a number field with ring of integers $\O_K$, and let $0\neq D\in\O_K$ be an algebraic integer. Then the set of elements in $V_n(\cO_K)$ with discriminant $D$ breaks up into finitely many $\GL_3(\cO_K)$-orbits.
\end{theorem}

The significance of the discriminant is as follows. Let $k$ be a field, and let $f\in V_n(k)$ be a nonzero element. We denote the curve cut out by $f$ in $\P^2_k$ by $C_f:=\Proj\; k[x,y,z]/(f)$. Then $C_f$ is a smooth curve (of genus $(n-1)(n-2)/2$) if and only if $\Delta_n(f)$ is nonzero. Let $K$ be a number field, let $\pp$ be a finite prime in $\O_K$, and let $\O_{\pp}$ denote the ring of integers of $K_{\pp}$. We say that an element $f\in V_n(\O_{\pp})$ has {\it good reduction} at $\frak{p}$ if $\Delta_n(f)$ is a unit in $\O_{\pp}$ (equivalently, if $C_{\overline{f}}$ is a smooth curve over $\O_{\pp}/\pp$, where $\overline{f}$ denotes the reduction of $f$ modulo $\pp$). 
Given a finite set of places $S$ of $K$, we denote the ring of $S$-integers in $K$ by $\O_{K,S}$. Then we have the following result.

\begin{theorem}\label{thm2}
Let $n\geq 2$ be a fixed integer. Let $K/\Q$ be a number field and let $S$ be a finite set of places of $K$ containing all the infinite places. Then the set of elements in $V_n(\O_{K,S})$, with good reduction outside $S$, breaks up into finitely many $\GL_3(\O_{K,S})$-orbits.
\end{theorem}

We make the following remarks.
\begin{remark}{\rm 
{\bf (1).} Our two main results are analogues of Theorems 1 and 2 from \cite{BM}, with binary forms replaced by ternary forms and $\GL_2$ replaced with $\GL_3$. Our proof follows a somewhat similar structure, where we use Faltings' theorem as a diophantine input instead of finiteness of the unit equation used by Birch--Merriman.

\vspace{.05in}

\noindent {\bf (2).} As in the case of binary forms, finiteness of the set of orbits when {\it all} the invariants are fixed follows by reduction theory methods (carried out in great generality by Borel--Harishchandra \cite{BHC}). The novelty of the results of Birch--Merriman and ours is that only one particular invariant, namely the discriminant, needs to be fixed to ensure finiteness.

\vspace{.05in}

\noindent {\bf (3).} In our results, we consider $\GL_3(\cO_K)$-orbits. However, finiteness also holds in Theorem \ref{thm1} for $\SL_3(\cO_K)$ orbits, as will be clear from the proof. For Theorem \ref{thm2}, finiteness does not hold for $\SL_3(\O_{K,S})$ orbits. Indeed, if $f\in V_n(O_{K,S})$ has good reduction outside $S$, and $q\in\cO_K$ is an element divisible only by primes in $S$, then $q^k f$ also has good reduction outside $S$ for all $k$. And these elements are all pairwise $\SL_3(\cO_{K,S})$-inequivalent, since they have different disriminants.

}\end{remark}

Our next finiteness result concerns the irreducible representation $V_{2,2}$ of $\GL_3$ with highest weight $(4,2)$. For any ring $R$, the set of $R$-points of this representation is given by
\begin{equation*}
V_{2,2}(R) = \bigl(\Sym^2(R^3)\otimes\Sym^2(R^3)\bigr)/I,
\end{equation*}
where, if $\Sym^2(R^2)\otimes\Sym^2(R^2)$ is written as bihomogeneous degree-$(2,2)$ polynomials in $x_1,x_2,x_3$ and $z_1,z_2,z_3$, then $I$ is the ideal of all multiples of $\sum x_iz_i$. That is, we have
\begin{equation*}
I = \Bigl(\sum_{i=1}^3x_iz_i\Bigr)\cdot k[x_1,x_2,x_3]\otimes k[z_1,z_2,z_3]\cap k[x_1,x_2,x_3]_2\otimes k[z_i,z_2,z_3]_2.
\end{equation*}
The action of $\GL_3(R)$ on $V_{2,2}(R)$ is given by
\begin{equation*}
\gamma\cdot f(x_1,x_2,x_3,z_1,z_1,z_3) = f((x_1,x_2,x_3)\cdot \gamma,
(z_1,z_2,z_3)\cdot\delta(\gamma)),
\end{equation*}
where $\delta(\gamma)$ denotes the cofactor matrix of $\gamma$.

Let $k$ be a field of characteristic not equal to $2$. We explain in \S2.2 that an elements $v$ in $V_{2,2}(k)$ cuts out a subvariety $X_v$ in the flag variety $F_k=F(0,1,2;k)$. The two projection maps $F_k\to\P^2_k$ restrict to two maps $X_v\to\P^2_k$. It is known that generically, this subvariety is a K3 surface and that the two maps  are (degree-$2$) covering maps to $\P^2$. We say that an element $v\in V_{2,2}(k)$ is {\it generic} if $X_v$ is a K3 surface in $F_k$, and the two maps $p_1,p_2$ from $X_v$ to $\P^2_k$ are covering maps. If $K$ is a number field and $\pp$ is a finite prime of $K$ with norm relatively prime to $2$, we say that $v\in V_{2,2}(\cO_{K,\pp})$ has {\it good reduction} at $\pp$ if $\bar{v}\in V_{2,2}(\F_\pp)$ is generic. We will also restrict ourselves to elements $v\in V_{2,2}(\cO_K)$ such that $X_v$ has Picard rank $2$ over the complex numbers. This last condition can be easily enforced by picking a prime $\pp$ of $K$, and restricting to those elements $v\in V_{2,2}(\cO_K)$ such that the reduction $\bar{v}$ of $v$ satisfies $X_{\bar{v}}$ having Picard rank $2$ (see the discussion above Theorem 1 in \cite{FC}).

We then have the following result.

\begin{theorem}\label{thm22}
Let $K/\Q$ be a number field and let $S$ be a finite set of places of $K$ containing all the infinite places and all the places above $2$. Then the set of elements in $V_{2,2}(\O_{K,S})$, with good reduction outside $S$, and such that $X_v$ has Picard rank $2$, breaks up into finitely many $\GL_3(\O_{K,S})$-orbits.
\end{theorem}

We introduce a new framework allowing us to obtain finiteness results on orbits of representations from certain classical diophantine results: Let $V$ be a finite dimensional representation of an algebraic group $G$, and let $\Delta\in\Z[V]$ be a polynomial relative invariant for the action of $G$ on $V$. Suppose for simplicity we are trying to prove that there are finitely many $G(\Z)$-orbits on $V(\Z)$ with invariant $\Delta$ fixed to be some nonzero integer. This can be done with the following four ingredients:
\begin{enumerate}
\item For fields $k$, we require $G(k)$-orbits on $V(k)$ to correspond to some geometric objects (say some sort of varieties over $k$). Moreover, we want $\Delta=0$ to correspond to objects that are degenerate in some way.
\item A diophantine input, which proves that the set of isomorphism classes of these objects over $\Q$ with non-degenerate reduction outside a fixed finite set of places is finite.
\item A geometric input, which proves that elements in $V(\C)$ with the same invariant $\Delta$ and corresponding to isomorphic geometric objects lie inside finitely many $G(\C)$-orbits.
\item A geometry-of-numbers input, which proves that a $G(\C)$-equivalent set of elements in $V(\Z)$ with fixed invariant $\Delta$ must break up into finitely many $G(\Z)$-orbits.
\end{enumerate}

Birch and Merriman's proof of their main result is easy to write in this framework: nonzero binary $n$-ic forms with coefficients in $\Q$ give rise to dimension-$0$ degree-$n$ subschemes of $\P^1_\Q$, yielding degree-$n$ \'etale algebras over $\Q$. Hermite's theorem implies that the number of possibilities for these \'etale algebras is finite, once the discriminant of the binary $n$-ic forms has been fixed.
The finiteness of the unit equation gives us Step 3, namely, that there are finitely many $\GL_2(\C)$-equivalence classes (in fact $\GL_2(\overline{\Q})$-equivalence classes) on the set of binary $n$-ic forms with fixed discriminant and corresponding to the same \'etale algebra over $\Q$. Indeed, the binary $n$-ic form factors over $\overline{\Q}$, and the action of $\SL_2(\overline{\Q})$ can be used to move three of its roots to $0$, $1$, and $\infty$ in $\P^1_{\overline{Q}}$. Since the discriminant is fixed, it is easy to see that the remaining roots have only finitely many possibilities -- they correspond to solutions to a unit equation. Finally, Step 4 can be deduced from an application of Borel--Harish-Chandra \cite{BHC}. (In fact Borel--Harish-Chandra can be used to obtain Step 4 whenever the following criteria apply: 1. $G$ is a reductive group; 2. elements with $\Delta\neq 0$ are stable (in the sense of GIT); and 3. the stabilizer in $G(\C)$ of generic elements of $V(\C)$ is finite.)

Our proof of Theorem \ref{thm1} (on the representations $V_n$ of $\GL_3$) is also easy to summarize in this framework. Assume that $n\geq 3$. (The case $n=1$ is trivial, and the case $n=2$ follows from the therory of quadratic forms.) For a field $k$, a $\GL_3(k)$-orbit on $V_n(k)$ yields a degree-$n$ curve in $\P^2_k$, which is smooth if and only if the discriminant is nonzero. Our diophantine input is Faltings' theorem when $n\geq 4$ and Siegel's theorem on the finiteness of integral points on elliptic curves when $n=3$. It is easy to prove that for $n\geq 4$, two complex ternary $n$-ic forms forms cutting out isomorphic smooth curves in $\P^2_\C$ are $\GL_3(\C)$-equivalent, yielding Step 3. When $n=3$, we obtain Step 3 directly from the invariant theory of ternary cubic forms. Step 4 once again is an application of Borel--Harish-Chandra. To prove Theorem \ref{thm2} only requires a small modification of this framework. The issue is that there are possibly infinitely many values for the discriminant of a ternary form with good reduction outside a fixed finite set of places. It is necessary to use the $\G_m$-action to cut this down to a finite set. The specific results required for this modified framework to work are proved in \S3.2.

The proof of Theorem \ref{thm22} proceeds as follows: The Borel--Weil realization of $V_{2,2}$ implies that generic $\GL_3(k)$-orbits on $V_{2,2}(k)$ yield K3 surfaces $X$ in the flag variety $F_k$, along with two double coverings $X\to\P^2_k$. For our diophantine input, we use Faltings' theorem to obtain that the isomorphism classes of the two branch curves of these two maps to $\P^2$ have only finitely many choices. Our main technical input is Step 3: we prove that the set of K3 surfaces $X$ in $F_\C$ such that the isomorphism classes of the branch covers of the two maps $X\to\P^2_\C$ are fixed is finite. This is accomplished in \S4.1. The rest of the proof follows identically to the proof of Theorem \ref{thm2}.

\subsection*{Acknowledgments}
We are very grateful to Simone Coccia, Ananth Shankar, Jacob Tsimerman, and Charlie Wu for many helpful conversations. We are very grateful to an anonymous referee for pointing out a hole in the proof of Proposition \ref{prop:tau_choices} in a previous version of this article. The second named author was supported by an NSERC grant and a Simons fellowship.

\section{The representations $V_n$ and $V_{2,2}$ of $\GL_3$}

In this section, we collect some preliminary results on the geometry and invariant theory of the representations $V_n$ and $V_{2,2}$ of $\GL_3$. Specifically, we associate curves to elements in these representations, and introduce polynomial invariants which determine smoothness of these geometric objects.

\subsection{Plane curves corresponding to elements in $V_n$}

Let $R$ be a ring. Recall that $V_n(R)$ denotes the set of homogeneous degree-$n$ polynomials in $R[x,y,z]$, and that $\GL_3(R)$ acts on $V_n(R)$ via linear change of variables. To any element $f\in V_n(R)$, we associate the scheme $C_f=C_{f,R}:=\Proj(R[x,y,z]/(f)$. When $R=k$ is a field, and $f\in V_n(k)$, and $C_{f,k}$ is a smooth plane curve in $\P_k^2$ of genus $(n-1)(n-2)/2$.

The {\it discriminant} $\Delta=\Delta_n\in\Z[V_n]$ is a $\GL_3(\Z)$-invariant on $V_n(\Z)$ defined by
\begin{equation}
\Delta(f):=c_n\mathrm R\Bigl(
\frac{\partial f}{\partial x}, \frac{\partial f}{\partial y},
\frac{\partial f}{\partial z}
\Bigr),
\end{equation}
where $\mathrm R$ above denotes the resultant of the three partial derivatives of $f$, and $c_n$ is a nonzero rational constant, unique up to sign, ensuring that the coefficients of $\Delta_n$ are relatively prime. We collect the following well known properties of the discriminant polynomial.
\begin{itemize}
\item [{\rm (1)}] The polynomial $\Delta_n$ is irreducible and homogeneous of degree $3(n-1)^2$.
\item [{\rm (2)}] For a field $k$ and $f\in V_n(k)$, we have that $C_{f,k}$ is a smooth plane curve if and only if $\Delta_n(f)\neq 0$.
\end{itemize}
For a comprehensive treatment of discriminant polynomials, see \cite[Chapter 13.1 D]{gelfand}.

\subsection{The Borel--Weil realization of $V_{2,2}$}

The Borel--Weil theorem realizes finite dimensional irreducible representations of reductive groups as spaces of sections on certain line bundles on flag varieties. In this subsection, we describe this realization for the representation $V_{2,2}$ of $\GL_3$.

For a field $k$, let $F_k= F(0,1,2,k)$ denote the complete flag variety over $k$: we have
\begin{equation*}
F(k)=\{(\ell,p):\ell\subset p\},
\end{equation*}
the set of pairs $(\ell,p)$, where $\ell\subset k^3$ is a line going through the origin and $p\subset k^3$ is a plane containing $\ell$. The action of $\GL_3(k)$ on $k^3$ yields an action of $\GL_3(k)$ on $F(k)$.

Clearly $F_k$ admits an embedding into the product of Grassmanians
\begin{equation*}
F_k\hookrightarrow G(1,3) \times G(2,3),
\end{equation*}
where $G(m,n)$ denotes the Grassmannian of $m$-dimensional subspaces of $n$-space over $k$. Both these Grassmannians are isomorphic to $\P^2_k$, and so we obtain two projection maps
\begin{equation*}
p_1,p_2:F_k\to\P^2_k.
\end{equation*}
Thees projection maps send the pair $(\ell,p)$ to $\ell$ and to $p$, respectively. (Here we are using the fact that $\P^2_k$ is isomorphic to its dual.)

Consider the line bundle $\L$ of the flag variety $F_k$, defined by
\begin{equation*}
\L:=(p_1)^*(\O(2))\otimes(p_2)^*(\O(2)).
\end{equation*}
The Borel--Weil theorem realizes the representation $V_{2,2}(k)$ of $\GL_3(k)$ as the space of global sections of $\L$ as follows.

We will parametrize lines in $k^3$ going through the origin via points $[x_1:x_2:x_3]\in\P^2(k)$. We will parametrize planes in $k^3$ going through the origin via points in $\P^2(k)$, where $[z_1:z_2:z_3]\in\P^2(k)$ corresponds to the plane given by the equation $\sum z_ix_i=0$. Therefore, the flag variety $F_k$ is given by the hypersurface in $\P^2_k\times\P^2_k$ cut out by the equation $\sum x_iz_i=0$. It is easy to check that the action of $\GL_3(k)$ on $F(k)$ descends from the following action of $\GL_3(k)$ on $\P^2_k\times\P^2_k$:
\begin{equation*}
\gamma\cdot ([x_1:x_2:x_3],[z_1:z_2:z_3])=( [x_1:x_2:x_3]\gamma,[z_1:z_2:z_3]\delta(\gamma)),
\end{equation*}
where $\delta(\gamma)=\det(\gamma)(\gamma^{-1})^t$ denotes the cofactor matrix of $\gamma$.
In this setup, $H^0(F_k, \L)$ is a quotient of the tensor product of the degree-$2$ parts of the polynomial rings $k[x_1, x_2, x_3]$ and $k[z_1, z_2, z_3]$. Specifically, we have: 
\begin{equation*} 
H^0(F_k, \L) = \frac{k[x_1, x_2, x_3]_2 \otimes k[z_1, z_2, z_3]_2}{I}, \end{equation*} where $I$ is given by
\begin{equation*} 
I:=\Bigl(\sum_{i=1}^3 x_i z_i\Bigr)\cap k[x_1, x_2, x_3]_2 \otimes k[z_1, z_2, z_3]_2.
\end{equation*}
The group $\mathrm{GL}_3(k)$ acts naturally on $H^0(F_k, \L)$. For $\gamma \in \mathrm{GL}_3(k)$ and a bihomogeneous $(2,2)$ polynomial $f(x_i,z_i)$, the action is given by:
\begin{equation*} \gamma \cdot f\big((x_1, x_2, x_3), (z_1, z_2, z_3)\big) = f\big((x_1, x_2, x_3)\gamma, (z_1, z_2, z_3)\delta(\gamma)\big). \end{equation*}
To see that this is an action on $H^0(F_k,\L)$, it is only necessary to check that the ideal generated by $x_1z_1+x_2z_2+x_3z_3$ is fixed by this action. This follows by noting that we have
\begin{equation*}
\sum_{i=1}^3x_iz_i = (x_1,x_2,x_3)(z_1,z_2,z_3)^t=
\frac{1}{\det(\gamma)}(x_1,x_2,x_3)\gamma((z_1,z_2,z_3)\delta(\gamma))^t,
\end{equation*}
for every $\gamma\in\GL_3(k)$. The content of the Borel-Weil theorem is that under this action, we have an isomorphism of the $\GL_3(k)$-modules $H^0(F_k,\L)$ and $V_{2,2}(k)$.

\subsection{K3 surfaces corresponding to elements in $V_{2,2}(k)$, and their branch curves}

Let $k$ be a field with characteristic not $2$. Given $v\in V_{2,2}(k)$, we let $S_v$ denote the surface in $F_k$ cut out by the corresponding section of $\L$. It is well known that, generically, $S_v$ is a K3 surface. Furthermore, this construction of $S_v$ from $v$ respects the two actions of $\GL_3(k)$ on $F(k)$ and $V_{2,2}(k)$. That is, for $\gamma\in \GL_3(k)$, we have $S_{\gamma\cdot v}=\gamma(S_v)$.
Let $v\in V_{2,2}(k)$ be an element which corresponds to the K3 surface $S_v$. Denote the branch curves of the two maps $p_1:S_v\to\P^2_k$ and $p_2:S_v\to \P^2_k$ by $C_{1,v}$ and $C_{2,v}$. When both $p_1$ and $p_2$ are covering maps, 
$C_{1,v}$ and $C_{2,v}$ are degree-$6$ curves in $\P^2_k$. We have the following lemma (see, for example \cite[Example 1.3, iv]{K3Global}).

\begin{lemma}
Let notation be as above. Assume that $p_1$ $($resp.\ $p_2)$ are covering maps from $S_v$ to $\P^2_k$. Then $S_v$ is smooth if and only if $C_{1,v}$ $($resp.\ $C_{2,v})$ is smooth.
\end{lemma}

We now construct two relative covariants for the representation $V_{2,2}$ of $\GL_3$. These covariants are defined over $\Z$, and so we work in a general setting. Given a degree-$(2,2)$ bihomogeneous polynomial $f(x_1,x_2,x_3,z_1,z_2,z_3)$, with coefficients in a ring $R$, we consider it as a ternary quadratic form $Q_z(x_1,x_2,x_3)$ in the $x_i$, whose coefficients are degree-$2$ homogeneous polynomials in the $z_i$. We abuse notation, and use $Q_z$ to denote the $3\times 3$ symmetric Gram matrix associated to this quadratic form. Then we have
\begin{equation*}
f(x_1,x_2,x_3,z_1,z_2,z_3)=(x_1,x_2,x_3)Q_z(x_1,x_2,x_3)^t.
\end{equation*}
Similarly, we may consider $f$ as a ternary quadratic form $Q_x(z_1,z_2,z_3)$ in the $z_i$ whose coefficients are degree-$2$ homogeneous polynomials in the $x_i$. Once again, we use $Q_x$ to also denote the Gram matrix associated to this quadratic form, and we have
\begin{equation*}
f(x_1,x_2,x_3,z_1,z_2,z_3)=(z_1,z_2,z_3)Q_x(z_1,z_2,z_3)^t.
\end{equation*}
Then our two covariants are given by
\begin{equation*}
I_x(f):=(x_1,x_2,x_3)\Adj(Q_x)(x_1,x_2,x_3)^t,\quad\quad
I_z(f):=(z_1,z_2,z_3)\Adj(Q_z)(z_1,z_2,z_3)^t,
\end{equation*}
where $\Adj(Q)$ denotes the adjoint matrix of $Q$.
We next prove that these quantities are well defined, and in fact covariants.
\begin{proposition}\label{prop_const_covariants}
Let $R$ be any ring. The maps 
\begin{equation*}
I_x:V_{2,2}(R)\to V_6(R)\quad\mbox{ and }\quad I_z:V_{2,2}(R)\to V_6(R)
\end{equation*}
are well defined. Moreover, we have 
\begin{equation*}
I_x(\gamma f)=\det(\gamma)^2\gamma(I_x(f)),\quad\quad 
I_z(\gamma f)=\delta(\gamma)(I_z(f)),
\end{equation*}
for all $\gamma\in \GL_3(R)$.
\end{proposition}
\begin{proof}
To prove the first claim, we must check that $I_x(f)$ only depends on the class of $f$ in $V_{2,2}(R)$. That is, $I_x(f)=I_x(f+g)$ for any $g$ of the form $L(\sum_{i=1}^3x_iz_i)$ with $L$ a degree-$(1,1)$ bihomogeneous polynomial. (Once this is proven for $I_x$, the analogous claim will follow for $I_z$ by symmetry.) We prove this claim via a direct computation. Denote the $ij$th coefficients of the ternary quadratic form in the $z_i$ corresponding to $f$ by $M_{ij}$. (Note that the $M_{ij}$ are ternary quadratic forms in $x_1$, $x_2$, and $x_3$.) Let $L=2(T_1z_1+T_2z_2+T_3z_3)$, where the $T_i$ are linear forms in the $x_i$. Then the Gram matrix of the quadratic form (in the $z_i$) corresponding to $g$ is
\begin{equation*}
\begin{pmatrix}
T_1 \\
T_2 \\
T_3
\end{pmatrix}
(x_1\,x_2\,x_3)
+
\begin{pmatrix}
x_1 \\
x_2 \\
x_3
\end{pmatrix}
(T_1\,T_2\,T_3)
=
\begin{pmatrix}
2x_1 T_1 & x_1 T_2 + x_2 T_1 & x_1 T_3 + x_3 T_1 \\
x_1 T_2 + x_2 T_1 & 2x_2 T_2 & x_2 T_3 + x_3 T_2 \\
x_1 T_3 + x_3 T_1 & x_2 T_3 + x_3 T_2 & 2x_3 T_3
\end{pmatrix}.
\end{equation*}
It is necessary for us to check that
\begin{equation*}
(x_1\,x_2\,x_3)
\Adj\begin{pmatrix}
2x_1 T_1 + M_{11} & x_1 T_2 + x_2 T_1 + M_{12} & x_1 T_3 + x_3 T_1 + M_{13} \\
x_1 T_2 + x_2 T_1 + M_{12} & 2x_2 T_2 + M_{22} & x_2 T_3 + x_3 T_2 + M_{23} \\
x_1 T_3 + x_3 T_1 + M_{13} & x_2 T_3 + x_3 T_2 + M_{23} & 2x_3 T_3 + M_{33}
\end{pmatrix}
\begin{pmatrix}
x_1 \\
x_2 \\
x_3
\end{pmatrix}
\end{equation*}
does not depend on $T_1$, $T_2$, and $T_3$.
It is clear that the term $T_1^2$. By symmetry, it is then enough to check that the coefficients of $T_1$ and $T_1T_2$ are $0$. An easy (though somewhat tedious) calculation, thus yields the first claim.

To demonstrate the covariance property, consider the action of $\gamma$ on the form $f$, expressed as:
\begin{equation*}
    \gamma \cdot f = (z_1, z_2, z_3) \cdot (\delta(\gamma) Q_x^\gamma \delta(\gamma)^t) \cdot (z_1, z_2, z_3)^t,
\end{equation*}
where $Q_x^\gamma$ denotes the matrix obtained by applying $\gamma$ to the entries of $Q_x$, the matrix associated with $f$. Now, examine $I_x(\gamma \cdot f)$:

\begin{align*}
    I_x(\gamma \cdot f) &= (x_1, x_2, x_3) \cdot \operatorname{Adj}(Q_x(\gamma \cdot f)) \cdot (x_1, x_2, x_3)^t  \\
    &= (x_1, x_2, x_3) \cdot \operatorname{Adj}(\delta(\gamma) Q^{\gamma}_x \delta(\gamma)^t) \cdot (x_1, x_2, x_3)^t \\
    &= (x_1, x_2, x_3) \cdot \operatorname{Adj}(\delta(\gamma)^t) \cdot \operatorname{Adj}(Q^{\gamma}_x) \cdot \operatorname{Adj}(\delta(\gamma)) \cdot (x_1, x_2, x_3)^t \\
    &= \det(\gamma)^2 \gamma(I_x(f)).
\end{align*}
Similarly, for $\gamma \cdot f$, we obtain  
\begin{equation*}
    \gamma \cdot f = (x_1, x_2, x_3) \cdot (\gamma Q_z^{\delta(\gamma)} \gamma^t) \cdot (x_1, x_2, x_3)^t,
\end{equation*}
where $Q_z^{\delta(\gamma)}$ is obtained by applying $\delta(\gamma)$ to the entries of the matrix $Q_z$ associated with $f$. Now, consider $I_z(\gamma \cdot f)$:
\begin{align*}
    I_z(\gamma \cdot f) &= (z_1, z_2, z_3) \cdot \operatorname{Adj}(Q_z(\gamma \cdot f)) \cdot (z_1, z_2, z_3)^t \\
    &= (z_1, z_2, z_3) \cdot \operatorname{Adj}(\gamma Q_z^{\delta(\gamma)} \gamma^t) \cdot (z_1, z_2, z_3)^t \\
    &= (z_1, z_2, z_3) \cdot \operatorname{Adj}(\gamma^t) \cdot \operatorname{Adj}(Q^{\delta(\gamma)}_z) \cdot \operatorname{Adj}(\gamma) \cdot (z_1, z_2, z_3)^t \\
    &= \delta(\gamma)(I_z(f)).
\end{align*}
This completes the proof.
\end{proof}

The significance of these two covariants is in the following result:

\begin{proposition}\label{prop_branch_covariants}
Let $k$ be a field of characteristic not $2$, and let $f\in V_{2,2}(k)$ correspond to a smooth K3 surface such that both projection maps $S_f\to\P^2$ are covering maps. Then the two branch curves $C_{1,f}$ and $C_{2,f}$ of these covering maps are cut out by the covariants $I_z(f)$ and $I_x(f)$, respectively.
\end{proposition}
\begin{proof}
Let $\overline{k}$ denote the algebraic closure of $k$. The surface $S_f(\overline{k})$ is cut out in $\P^2(\overline{k})\times \P^2(\overline{k})$ by the equations 
\begin{equation*}
f(x_1,x_2,x_3,z_1,z_2,z_3)=0,\quad x_1z_1+x_2z_2+x_3z_3=0.
\end{equation*}
 A point $P=[a_1:a_2:a_3]\in\P^2(\overline{k})$ lies on $C_{1,f}(\overline{k})$ if and only if the line $L_P$ cut out by $\sum a_iz_i=0$ is tangent to the conic cut out by $Q_{(a_1,a_2,a_3)}$, where we recall that we write $f(x_i,z_i)=(z_1,z_2,z_3)Q_x(z_1,z_2,z_3)^t$.
 By definition, a line is tangent to a conic if and only if the line lies on the dual conic. Therefore $L_P$ being tangent to the conic cut out by $Q_{(a_1,a_2,a_3)}$ is equivalent to the equation
 \begin{equation*}
(a_1,a_2,a_3)\Adj(Q_{(a_1,a_2,a_3)})(a_1,a_2,a_3)^t = I_z(a_1,a_2,a_3)=0
 \end{equation*}
 being satisfied,
 since the adjoint matrix of a symmetric matrix $Q$ cuts out the conic dual to the conic cut out by $Q$. This concludes the proof of the first part of the proposition. The proof of the second part is identical.
\end{proof}

\section{Finiteness results for ternary forms}

In this section, we prove Theorems \ref{thm1} and \ref{thm2}.

\subsection{Ternary cubic forms}

Let $K$ be a number field. When $n=2$, our main theorems follow from the theory of quadratic forms. In this subsection, we consider the space $V_3$ of ternary cubic forms. It is well known (see, for example, \cite{ternary}) that the ring of polynomial invariants for the action of $\SL_3(\Z)$ on $V(\Z)$ is freely generated by two invariants, usually denoted by $I$ and $J$. Here $I$ has degree $4$ and $J$ has degree $6$, and the discriminant $\Delta=\Delta_3$ is given by
\begin{equation*}
\Delta(f) := \Delta(I(f),J(f)):=\frac{1}{27}(4I(f)^2-J(f)^3).
\end{equation*}
Let $K$ be a number field, and let $S$ be a finite set of places of $K$ containing all the infinite places. Our main diophantine input for this subsection is Siegel's theorem on the finiteness of $S$-integral points of elliptic curves.
\begin{theorem}[\cite{Siegel}]\label{th:Siegel}
Let $E$ be an elliptic curve over a number field $K$. Let $S$ be a finite set of places of $K$, containing all the infinite places. Then the number of $S$-integral points on $E$ is finite.
\end{theorem}

We will combine this diophantine input with the following well known application of the construction of fundamental domains for the action of $\SL_3(\O_{K,S})$ on $\prod_{v\in S}\SL_3(K_v)$ (see works of Siegel \cite{Siegel} and Borel--Harish-Chandra's \cite{BHC}). 
\begin{theorem}\label{th:BHC}
Fix $n\geq 3$. Let $K$ and $S$ be as above, and let $x$ be an element of $V_n(\cO_{K,S})$ which is stable in the sense of GIT. Then there are finitely many $\SL_3(\cO_{K,S})$-orbits on $V_n(\cO_{K,S})\cap (\prod_{v\in S}\SL_3(K_v)\cdot x)$.
\end{theorem}

The above result has the following immediate corollary.

\begin{corollary}\label{cor:BHC}
Let $T\subset V_n(\cO_{K,S})$ be a subset of stable elements such that for $f_1, f_2\in T$, we have $P(f_1)=P(f_2)$ for {\it every} polynomial $\SL_3(\Z)$-invariant $P\in \Z[V_n]$. Then $T$ breaks up into finitely many $\SL_3(\O_{K,S})$-equivalence classes.
\end{corollary}
\begin{proof}
Fix $n\geq 3$. Given Theorem \ref{th:BHC}, it is only necessary to prove that for any place $v$, the set $T$ breaks up into finitely many $\SL_3(K_v)$-equivalence classes. Since every element in $T$ is stable and has the same set of invariants, it follows that $T$ lies within a single $\SL_3(\overline{K_v})$-orbit, say the orbit of some $x\in V_n(K_v)$. Let $\Gamma$ denote the stabilizer of $x$ in $\SL_3(\overline{K_v})$, and note that $\Gamma$ is a finite group.
It is well known (see for example, \cite[\S I.5]{Serre_Gal_coh}) that we have an injection
\begin{equation*}
\SL_3(K_v)\backslash\bigl(\SL_3(\overline{K_v})\cdot x\cap V_n()\bigr)
\to H^1(G_{\overline{K_v}/K_v},\Gamma).
\end{equation*}
Since there are finitely many fixed degree \'etale extensions of $K_v$, this Galois cohomology group is finite, yielding the result.
\end{proof}

We are now ready to prove Theorems \ref{thm1} and \ref{thm2} in our case.

\medskip

\noindent{\bf Proof of Theorems \ref{thm1} and \ref{thm2} for $n=3$:}
Let $D\in\cO_K$ be nonzero. Then from Theorem \ref{th:Siegel}, we see that there are only finitely main pairs $(I,J)\in\O_K^2$ with $\Delta(I,J)=D$. Theorem \ref{thm1} now follows from a direct application of Corollary \ref{cor:BHC}, taking $S$ to be the set of infinite valuations.

We turn to Theorem \ref{thm2}. Let $S$ be a finite set of places, containing all the infinite places. Let $f$ be an element of $V(\O_{K,S})$ such that $C_f$ has good reduction outside $S$. Then $\Delta(f)$ is an element of $\cO_{K,S}^\times$, an $S$-unit. We first claim that by translating $f$ by an element of $\GL_3(\cO_{K,S})$, if necessary, we may ensure that $\Delta(f)$ sits inside a finite set. Indeed, acting in $f$ by an element in the center of $\GL_3(\cO_{K,S})$, we may multiply $f$ by an $S$-unit $u$. This has the effect of multiplying the discriminant of $f$ by $u^{36}$. The claim follows since $\cO_{K,S}^\times/(\cO_{K,S}^\times)^{36}$ is finite. Now that we only have to consider discriminants in a finite set, Theorem \ref{thm2} follows by applying Theorem \ref{th:Siegel} and Corollary \ref{cor:BHC} exactly as above. $\Box$

\subsection{The proof for general $n\geq 4$}

Let $n\geq 4$ be fixed. Recall that for an element $f\in V_n(\C)$ having nonzero discriminant, $C_f$ is a smooth curve of genus $(n-1)(n-2)/2\geq 2$. Our main diophantine input in this case is Faltings' theorem.

\begin{theorem}\label{th:Faltings}
Let $K$ be a number field, and let $S$ be a finite set of places containing all infinite places. Fix an integer $g\geq 2$. Then there are finitely many smooth projective curves $C/K$ of genus $g$ with good reduction outside $S$.
\end{theorem}

To apply Faltings' theorem in our setting, we will also need the following well known result, whose proof we include for completeness.
\begin{proposition}\label{prop:curve_to_form}
Let $f_1$ and $f_2$ be elements in $V_n(\C)$ having nonzero discriminants. Let $C_1$ and $C_2$ denote the $($smooth$)$ plane curves cut out by $f_1$ and $f_2$, respectively, in $\mathbb{P}^2_\mathbb{C}$. Suppose that $C_1$ is isomorphic to $C_2$. Then $f_1$ and $f_2$ are $\mathrm{GL}_3(\mathbb{C})$-equivalent.
\end{proposition}
\begin{proof}
Let $C$ be a smooth plane curve over $\C$ of degree $n\geq 4$. It is known that $C$ has a unique $g_n^2$, i.e., a linear system of dimension $2$ and degree-$n$. Moreover, this $g_d^2$ of $C$ is $|\O_C(1)|$, which satisfies $h^0(C,\O_C(1))=3$. This is the content of \cite[Exercise 18]{ACGH}, and is proved in, for example, \cite[Proposition 2.6]{Landesman}. Hence, the set of embeddings of $C$ into $\P^2_\C$ are in bijection with the set of choices of an ordered basis for $H^0(C,\O_C(1))$, up to scaling. Therefore, the images of any two embeddings are $\GL_3(\C)$-equivalent, yielding the claim.
\end{proof}

We are now ready to prove Theorem \ref{thm1}:

\medskip

\noindent {\bf Proof of Theorem \ref{thm1}:} Let $D\in\O_K$ be a fixed nonzero element, and let $T(D)$ denote the set of elements in $V_n(\O_K)$ with discriminant $D$. For any $f\in T(D)$, the curve $C_f$ has good reduction away from the places dividing $D$. Hence, by Faltings' theorem, there are only a finite number of choices for the isomorphism class of $C_f/K$. Applying Proposition \ref{prop:curve_to_form}, we see that $T(D)$ breaks up into finitely many $\GL_3(\C)$-equivalence classes.
Moreover, for $\gamma\in\GL_3(\C)$ and $f\in V_n(\C)$, we have $\Delta(\gamma f)=(\det \gamma)^{n(n-1)^2}\Delta(\gamma)$. It therefore follows that in fact $T(D)$ breaks up into finitely many $\SL_3(\C)$-equivalence classes. Let $T(D)_1$ be one of these equivalence classes, and take $f_1$ and $f_2$ be two elements in $S(D)_1$. Then since $f_1$ and $f_2$ are $\SL_3(\C)$-equivalent, it follows that $P(f_1)=P(f_2)$ for every polynomial $\SL_3(\Z)$-invariant $P$. Thus, we may apply Corollary \ref{cor:BHC} (taking $S$ to be the set of infinite places of $K$) to conclude that $T(D)_1$ breaks up into finitely many $\SL_3(\O_K)$-orbits. Theorem \ref{thm1} now follows. $\Box$

\medskip

Before we prove Theorem \ref{thm2}, it will be convenient for us to set up some notation. Let $V$ be $V_n$ for some $n\geq 2$ or $V_{2,2}$. We denote the ring of $\SL_3(\Z)$-invariant integer polynomials in the coefficients of $V$ by $\Z[V]^{\SL_3(\Z)}$, which inherits the graded structure on $\Z[V]$. The action of $\GL_3$ on $V$ induces a natural action of $\G_m$ on $\Z[V]^{\SL_3(\Z)}$, such that for any ring $R$, any homogeneous $\SL_3(R)$-invariant polynomial $P\in V[R]$, any $v\in V(R)$, and any $\gamma\in\GL_3(R)$, we have $P(\gamma v)=(\det\gamma.P)(v)$. (Here, we are denoting the action of $\G_m(R)=R^\times$ on $R[V]^{\SL_3(R)}$ with the notation $\lambda.P$ for $\lambda\in R^\times$, and an invariant polynomial $P$.)

We denote the GIT quotient $V\sslash \SL_3=\spec \Z[V]^{\SL_3(\Z)}$ by $\I$. For a ring $R$, we have the map $\iota: V(R)\to \I(R)$, where we regard $\I(R)$ to be the space of invariants and $\iota$ as the map sending an element of $V(R)$ to its invariants.
Since the ring of invariants for the action of $\SL_3(\Z)$ on $V(\Z)$ is finitely generated, we can think of $\I_R$ as a subvariety, defined over $\Z$, of some affine space $\A^m_R$. The action of $R^\times$ on $R[V]^{\SL_3(R)}$ gives an action of $R^\times$ on $\I(R)$.

Let $K$ be a number field, and let $S$ be a finite set of places in $K$ containing all the archimedean places.  For any prime $\mathfrak{p}$ of $K$ not contained in $S$, the reduction modulo $\mathfrak{p}$ map $\I(\cO_{K,S})\to \I(\O_K/\mathfrak{p})$ is well defined. Let $\I(\O_{K,S})'\subset\I(\cO_{K,S})$ denote the set of elements in $\I(\cO_{K,S})$ which do not map to $0$ (the point $(0,0,\ldots,0)\in(\O_{K}/\pp)^m$) for every prime $\pp$ not in $S$. Note that $\I(\O_{K,S})'$ is $\O_{K,S}^\times$-invariant.
We have the following lemma:

\begin{lemma}\label{finite fiber pi}
The map $\pi$
\begin{equation*}
\O_{K,S}^\times\backslash\I(\O_{K,S})'=:\overline{\I(\O_{K,S})'}\xlongrightarrow{\pi}\overline{\I(\C)}:=\C^\times\backslash \I(\C)
\end{equation*}
has finite fibers.
\end{lemma}
\begin{proof}
Let $(I_1,\dots,I_m) \in \mathcal{I}(\mathcal{O}_{K,S})'$ be fixed and denote by $n_i$ the weighted degree of the invariant $I_i$, i.e., for $\lambda\in\C^\times$, we have $\lambda.I_i=\lambda^{n_i}I_i$. We aim to show that the set of tuples $(J_1, \dots, J_m) \in \mathcal{I}(\mathcal{O}_{K,S})'$ satisfying
\begin{equation*} 
    (\alpha^{n_1} I_1, \dots, \alpha^{n_m} I_m) = (J_1, \dots, J_m) \quad \text{for some} \quad \alpha \in \mathbb{C}
\end{equation*}
contains only finitely many elements under the action of $\mathcal{O}_{K,S}^\times$. 

Since each $J_i$ is an element of $\mathcal{O}_{K,S}$, it follows that $\alpha^{n_i}$ must lie in $K$, which implies that $\alpha^d \in K$, where $d = \gcd(n_1, \dots, n_m)$.

By the definition of $\mathcal{I}'(\mathcal{O}_{K,S})$, for any prime $\pp \notin S$, we have
\begin{equation*} 
    \min \{ v_{\pp}(I_1), \dots, v_{\pp}(I_m) \} = \min \{ v_{\pp}(J_1), \dots, v_{\pp}(J_m) \} = 0.
\end{equation*}
In particular, there exists some index $i$ such that $v_{\pp}(I_i) = 0$, which forces $  v_{\pp}(J_i) = n_i/d \cdot v_{\pp}(\alpha^d)$. Thus, we obtain $v_{\pp}(\alpha^d) \geq 0$.

Next, consider some index $k$ for which $v_{\pp}(J_k) = 0$. Under the given assumptions, we have $v_{\pp}(I_k) \geq 0$. Thus, we obtain
\begin{equation*} 
    v_{\pp}(I_k) + n_k/d \cdot v_{\pp}(\alpha^d) = 0.
\end{equation*}
Since both terms are nonnegative, it follows that $v_{\pp}(\alpha^d) = 0$, meaning that $\alpha^d$ is a unit in $\mathcal{O}_{K,S}^\times$. Finally, because the quotient group $\mathcal{O}_{K,S}^\times / (\mathcal{O}_{K,S}^\times)^d$ is finite, we conclude that there are only finitely many $\mathcal{O}_{K,S}^\times$-equivalence classes of such tuples.
\end{proof}

We are now ready to prove our second main result on ternary $n$-ic forms.

\medskip

\noindent {\bf Proof of Theorem \ref{thm2}:} Let $n\geq 4$ be fixed, let $K$ be a number field, and $S$ be a finite set of places of $K$ containing all the infinite places. Let $T(K,S)$ denote the set of elements $f$ in $V_n(\O_{K,S})$ such that $f$ has good reduction outside $S$. We say that two elements \( f \) and \( g \) in \( T_n(K; S) \) are  \( \sim_{\mathbb{C}} \)-equivalent if there exists an element \( \gamma \in \mathrm{GL}_3(\mathbb{C}) \) such that \( \gamma \cdot f = g \). This gives us a natural map
\begin{equation*}
  \begin{tikzcd}
    \GL_3(\O_{K,S}) \backslash T_n(K; S) \arrow{r}{\pi'}  & \sim_{\mathbb{C}} \backslash T_n(K; S).
  \end{tikzcd}
\end{equation*}
Since the elements in $T_n(K,S)$ correspond to curves that are smooth over $\C$ (indeed, their discriminants do not vanish modulo places not in $S$), it follows that the invariant map $\iota$ maps $T_n(K,S)$ into $\I(\O_{K,S})'$. Hence we have the commutative diagram
\begin{equation*}
  \begin{tikzcd}
    \GL_3(\O_{K,S}) \backslash T_n(K; S) \arrow{r}{\pi'} \arrow[swap]{d}{\iota} & \sim_{\mathbb{C}} \backslash T_n(K; S) \arrow{d}{\iota} \\
    \overline{\mathcal{I}(\mathcal{O}_{K, S})'} \arrow{r}{\pi} & \overline{\mathcal{I}(\mathbb{C})}
  \end{tikzcd}
\end{equation*}
where \( \overline{\mathcal{I}(\mathcal{O}_{K, S})'} \) and \( \overline{\mathcal{I}(\mathbb{C})} \) are as defined previously.

By an argument identical to the proof of Theorem \ref{thm1} above, it follows that $T_n(K,S)$ breaks up into finitely many $\GL_3(\C)$-equivalence classes. Hence the image of $\iota(\sim_{\mathbb{C}} \backslash T_n(K; S)$ in $\overline{\I(\C)}$ is finite. It follows from Lemma \ref{finite fiber pi} that the image of $\iota\big(\GL_3(\O_{K,S}) \backslash T_n(K; S)\big)$ in $\overline{\I(\O_{K,S})'}$ is also finite. Therefore, there exists a finite set of elements in $\I'\subset \I(\O_{K,S})$ such that every ternary $n$-ic form in $T_n(K,S)$ is $\GL_3(\cO_{K,S})$-equivalent to some $f\in T_n(K,S)$ satisfying $\iota(f)\in\I'$. Theorem \ref{thm2} now follows from a direct application of Corollary \ref{cor:BHC}. $\Box$

\section{Finiteness results for $\GL_3$-orbits on $V_{2,2}$}

In this section, we prove Theorem \ref{thm22}.

\subsection{Finiteness of generic K3 surfaces with fixed branch curves}

In this subsection, we work over the field of complex numbers. Recall that if $X$ is a generic K3 surface in $F_\C$, then for $i\in\{1,2\}$, the projections maps $p_i:F_\C\to\P^2_\C$ give two double covers $p_i:X\to\P^2_\C$. We are further assuming that the Picard rank of $X$ over $\C$ is $2$.  Let $C_1$ and $C_2$ be two smooth degree-$6$ curves in $\P^2_\C$. We let $\FF_{C_1,C_2}$ denote the set of generic K3 surfaces in $F_\C$ whose two branch curves $B_1$ and $B_2$ are $\GL_3(\C)$-equivalent to $C_1$ and $C_2$, respectively. It follows from Propositions \ref{prop_const_covariants} and \ref{prop_branch_covariants} that $\FF_{C_1,C_2}$ is closed under the natural action of $\GL_3(\C)$. The goal of this subsection is to prove the following result.

\begin{theorem}\label{thm:FFCC_fin}
For any pair of smooth degree-$6$ curves $C_1$ and $C_2$ in $\P^2_\C$, the set $\FF_{C_1,C_2}$ breaks up into finitely many $\GL_3(\C)$-orbits.
\end{theorem}

A double plane over $\C$ is a pair $(X,f)$, where $X$ is a smooth surface over $\C$, and $f:X\to\P^2_\C$ is a double cover. The branch curve of $p$ will then be a smooth curve in $\P^2_\C$. We say that two double planes $(X,f)$ and $(X',f')$ are {\it equivalent} if there exists an isomorphism $\phi:X\to X'$ such that $f'\circ \phi=f$. The following result is an immediate consequence of Catanese's work on the Chisini conjecture \cite{chisini}.

\begin{proposition}\label{prop:chisni}
Let $B$ be a smooth curve in $\P^2_\C$. Then there are finitely many equivalence classes of double planes over $\C$ with branch curve $B$. 
\end{proposition}
\begin{proof}
Let $(X, f)$ and $(X', f')$ be two double covers of $\mathbb{P}^2$ branched over the same curve $B$. Consider the map
\begin{equation*}
    \varphi' = (f')^{-1} \circ f: R \to R',
\end{equation*}
where $R$ and $R'$ denote the ramification divisors of $f$ and $f'$, respectively. Since $f|_R$ and $f'|_{R'}$ are both isomorphisms onto $B$, it follows that $\varphi'$ is an isomorphism.
Let $N_{R/X}$ and $N_{R'/X'}$ denote the normal bundles of $R$ in $X$ and $R'$ in $X'$, respectively.
From \cite[Proposition 2.5]{chisini}, it follows that the element 
\begin{equation*}
    \eta(X, X') = N_{R/X}^{-1}\otimes (\varphi')^* N_{R'/X'}
\end{equation*}
is always a $2$-torsion element of $\operatorname{Pic}(R)\cong\Pic(B)$. Moreover, the main theorem in \cite{chisini} states that $\eta(X,X')$ is trivial if and only if $(X, f)$ and $(X', f')$ are equivalent as doubled planes.
These two results realize the set of equivalence classes of doubled planes over $\C$ with branch curves $B$ as a principal homogeneous space of a subgroup of $\Pic(B)[2]$. Since $\Pic(B)[2]$ is finite, this completes the proof.
\end{proof}

We apply the above result in our setting. Let $C_1$ and $C_2$ be two smooth degree-$6$ curves in $\P^2_\C$. We say that two generic K3 surfaces $X$ and $X'$ in $\FF_{C_1,C_2}$ are {\it equivalent} if there exist elements $\gamma_1,\gamma_2,\gamma_1',\gamma_2'\in\GL_3(\C)$ such that the pair of double planes $(X,\gamma_i\circ p_i)$ and $(X',\gamma_i'\circ p_i)$ are branched at $C_i$, and equivalent for $i\in\{1,2\}$. Then we have the following immediate consequence of Proposition \ref{prop:chisni}:
\begin{lemma}\label{lem:FCC_finite_eq_classes}
Fix two smooth degree-$6$ curves $C_1$ and $C_2$ in $\P^2_\C$. Then the set $\FF_{C_1,C_2}$ consists of finitely many equivalence classes.
\end{lemma}

Let $X$ and $X'$ be two equivalent generic K3 surfaces in $F_\C$. Then we obtain two isomorphisms $\phi_1,\phi_2: X\to X'$, coming from the two equivalences between the doubled planes:
 \[\begin{tikzcd}
	& {\P^2} & \\
	X \arrow[ur, "{\gamma_1\circ p_1}"] \arrow[rr, shift left=2, "{\phi_1}"] \arrow[rr, shift right=2, swap, "{\phi_2}"] \arrow[dr, swap, "{\gamma_2\circ p_2}"] 
	& & X' \arrow[ul, swap, "{\gamma_1'\circ p_1}"] \arrow[dl, "{\gamma_2'\circ p_2}"] \\
	& {\P^2} &
\end{tikzcd}\]
We therefore obtain an automorphism $\tau:X\to X$ by setting $\tau:=\phi_1^{-1}\circ\phi_2$. We next prove that in fact $\tau$ has only finitely many possibilities. To this end, we begin with the following Lemma. 

\begin{lemma}\label{Lattice}
Let $X$ be a generic K3 surface in $F_\C$, and let $p_1, p_2 \colon X \to \P^2_\C$ denote the two natural projection maps. Define divisor classes $H_j := p_j^*\mathcal{O}_{\P^2}(1)$ for $j = 1, 2$. Then:
\begin{enumerate}
    \item[\textnormal{(a)}] The intersection pairing on $\Pic(X)$ satisfies
    \begin{equation*}
        H_1^2 = H_2^2 = 2, \qquad H_1 \cdot H_2 = 4.
    \end{equation*}

    \item[\textnormal{(b)}] The sublattice of $\Pic(X)$ generated by $H_1$ and $H_2$ is isomorphic to $\Pic(F)$ and has index dividing $4$ in $\Pic(X)$.
\end{enumerate}
\end{lemma}

\begin{proof} 
We begin with Part (a).
The intersection products in $X$ can be calculated within the cohomology ring $H^*(F,\Z)$ of $F$, yielding
\begin{equation*}
\begin{array}{rcl}
H_1^2&=&(\O_F(2,2),\O_F(1,0),\O_F(1,0));\\[.1in]
H_1\cdot H_2&=&(\O_F(2,2),\O_F(1,0),\O_F(0,1));\\[.1in]
H_2^2&=&(\O_F(2,2),\O_F(0,1),\O_F(0,1)).
\end{array}
\end{equation*}
In particular, the value of these intersection numbers are independent of $X$, and are therefore as given by \cite[Proposition 2.6]{Wehler}.

We move on to Part (b). Let $L_1$ and $L_2$ be generators of the lattice $\Pic(X)$. Define the integral quadratic form $f(x, y) := (xL_1 + yL_2) \cdot (xL_1 + yL_2)$ given by the intersection pairing. Since $H_1$ and $H_2$ are also linearly independent elements in $\Pic(X)$, there exists a $2 \times 2$ matrix $\gamma$ with integral entries such that
\begin{equation*}
\gamma \cdot f(x, y) = (xH_1 + yH_2) \cdot (xH_1 + yH_2) = 2(x^2 + 4xy + y^2).
\end{equation*}
Comparing discriminants on both sides and using the invariance of the discriminant under change of basis implies that $\det(\gamma)^2 \mid 48$, and hence $\det(\gamma) \mid 4$.
This concludes the proof of the lemma.
\end{proof}

We now have the following proposition.
\begin{proposition}\label{prop:tau_choices}
Let $X$ and $X'$ be two equivalent generic K3 surfaces in $F_\C$. Let $\tau\in\Aut(X)$ be the corresponding automorphism of $X$. Then $\tau$ has finitely many options.
\end{proposition}
\begin{proof}
First note that the equivalence of $X$ and $X'$ (with $\tau$ being the associated automorphism) implies the existence of elements $\gamma_1,\gamma_2,\gamma_1',\gamma_2'\in\GL_3(\C)$ and an isomorphism $\phi:X\to X'$ such that
\begin{enumerate}
\item The morphisms $\gamma_1\circ p_1$ and $\gamma_1'\circ p_1\circ\phi$ from $X$ to $\P^2_\C$ are the same, with branch curves $C_1$.
\item The morphisms $\gamma_2\circ p_2$ and $\gamma_2'\circ p_2\circ\phi\circ\tau$ from $X$ to $\P^2_\C$ are the same, with branch curve $C_2$.
\end{enumerate}
For $i\in\{1,2\}$, let $H_i$ denote the pullbacks $p_i^*\cO_{\P^2}(1)$ on $X$, and let $H_i'$ denote the pullbacks $p_i^*\cO_{\P^2}(1)$ on $X'$. We have $\phi^*H_1'$ and $H_1$ are linearly equivalent, and $\phi^*H_2'$ and $(\tau^{-1})^* H_2$ are linearly equivalent. (This follows from the definitions of the maps, and Facts 1 and 2 above, in conjunction with the fact that for any $\gamma\in\GL_3(\C)$, $\gamma^*\O_{\P_2}(1)$ is linearly equivalent to $\O_{\P^2}(1)$.)

The automorphism $\tau$ of $X$ gives an action $\tau^*$ on $\Pic(X)$, and moreover the kernel of the natural map $\Aut(X)\to\Aut(\Pic(X))$ is finite by \cite[Chapter 15, Theorem 2.6]{K3Global}. The action of $\tau^*$ on $\Pic(X)$ is determined by the values of $\tau^*(H_1)$ and $\tau^*(H_2)$. Therefore, it suffices to show that there are only finitely many possible possibilities for the pair $(\tau^*H_1,\tau^*H_2)$. We write
\begin{equation*}
(\tau^{-1})^*H_1=a_{11}H_1+a_{12}H_2,\quad (\tau^{-1})^*H_2=a_{21}H_1+a_{22}H_2,
\end{equation*}
for values $a_{ij}\in\Q$. By Part (b) of Lemma \ref{Lattice}, we know that the matrix $A=(a_{ij})_{1\leq i,j\leq 2}$ has determinant in $\pm\{1/4,1/2,1,2,4\}$ and coefficients in $\frac{1}4\Z$. Moreover, $\tau$ and $\tau^{-1}$, being automorphisms, preserve the intersection pairing on $\Pic(X)$. Consider the quadratic form $f(x,y) := (xH_1+yH_2)\cdot(xH_1+yH_2)$. We have $f(x,y) = 2(x^2+4xy+y^2)$ by Part (a) of Lemma \ref{Lattice}, and $A$ belongs to the orthogonal group of this quadratic form. 

To further constrain the possible values of $A$, we proceed as follows. Let $L_1'$ and $L_2'$ generate $\Pic(X')$. Since $\phi\colon X\to X'$ is an isomorphism, the pullbacks $L_1=\phi^*(L_1')$ and $L_2=\phi^*(L_2')$ generate $\Pic(X)$. By Lemma~\ref{Lattice}(b), there are integral $2\times2$ matrices $M,M'$ with $|\det(M)|,|\det(M')|\mid4$ such that
\begin{equation*}
M\begin{pmatrix}L_1\\L_2\end{pmatrix}
=
\begin{pmatrix}H_1\\H_2\end{pmatrix},\quad
M'\begin{pmatrix}L_1'\\L_2'\end{pmatrix}
=
\begin{pmatrix}H_1'\\H_2'\end{pmatrix}.
\end{equation*}
Applying $\phi^*$ to the second equation and combining with the first gives a matrix $N = M'M^{-1}\in\GL_2(\Q)$ such that
\begin{equation*}
N\begin{pmatrix}H_1\\H_2\end{pmatrix}
=
\phi^*\!\begin{pmatrix}H_1'\\H_2'\end{pmatrix}
=
\begin{pmatrix}H_1\\(\tau^{-1})^*H_2\end{pmatrix}.
\end{equation*}
Since we have
\begin{equation*}
N = \begin{pmatrix}1&0\\ a_{21}&a_{22}\end{pmatrix},\qquad m,n\in\Q,
\end{equation*}
it follows that $a_{22}=\det(N)$ has only finitely many possibilities. Combining this with the intersection pairing property $(\tau^{-1})^* H_2 \cdot (\tau^{-1})^* H_2 = 2$, additionally gives $f(a_{21},a_{22}) = 2$. Therefore, $a_{21}$ also has only finitely many possibilities.

Recall that $\det(A)$ is constrained to lie in a finite set. Once $a_{21}$ and $a_{22}$ are fixed, we note that the condition $f(a_{11},a_{12})=2$ (implied by the invariance of the intersection pairing under), constrains the pair $(a_{11},a_{12})$ to also belong to a finite set. Indeed, $f(x,y)$ is irreducible, and so the intersection of $f(x,y)=2$ with any line has only finitely many solutions $(x,y)$. Therefore, $A$ has only finitely many possibilities, and we are done.
\end{proof}

We are now ready to prove our main result.

\medskip

\noindent {\bf Proof of Theorem \ref{thm:FFCC_fin}:} From Lemma \ref{lem:FCC_finite_eq_classes}, we see that $\FF_{C_1,C_2}$ consists of finitely many equivalence classes. We will show that each equivalence class consists of at most finitely many $\GL_3(\C)$-orbits. Let $X$ be an element in $\FF_{C_1,C_2}$, and let $X'\in\FF_{C_1,C_2}$ be equivalent to $X$. Then we have the diagram
\[\begin{tikzcd}
	& {\P^2} & \\
	X \arrow[ur, "{\gamma_1\circ p_1}"] \arrow[rr, shift left=2, "{\phi\circ\tau}"] \arrow[rr, shift right=2, swap, "{\phi}"] \arrow[dr, swap, "{\gamma_2\circ p_2}"] 
	& & X' \arrow[ul, swap, "{\gamma_1'\circ p_1}"] \arrow[dl, "{\gamma_2'\circ p_2}"] \\
	& {\P^2} &
\end{tikzcd}\]
where $\phi:X\to X'$ is an isomorphism and $\tau$ has finitely many options. We will show that the $\GL_3(\C)$-orbit of $X'$ is determined by $\tau$.

We begin with the following simplifications. First, note that by replacing $C_1$ and $C_2$ by $\gamma^{-1}_1C_1$ and $\gamma^{-1}_2C_2$, respectively, we can assume that $\gamma_1=\gamma_2=1$. Second, by replacing $X'$ with $ \gamma_2' \cdot X'$, we may assume that $\gamma_2'=1$. We set $\gamma_1'=\gamma^{-1}$ to simplify notation. We have used up the $\GL_3(\C)$-action by changing $X'$, and it is now required to prove that there are finitely many possibilities for $X'$. Note that $X'$ is determined by $\gamma$, and also that $\tau$ has finitely many options (by Proposition \ref{prop:tau_choices}). We will prove that $\gamma$ is determined by $\tau$, and this will finish the proof.

Let $\alpha \in \P^2_\C \setminus C_1$. Then there are exactly two points of $X$ mapping to $\alpha$ (via $p_1$). We regard $X$ and $X'$ are closed subsets of $\P^2_\C\times\P^2_\C$, and denote these two points by $(\alpha, \beta)$ and $(\alpha, \beta')$. Applying the automorphism $\tau$ to these points yields
\begin{equation*}
 \tau((\alpha, \beta)) = (\alpha_1, \beta_1), \quad \tau((\alpha, \beta')) = (\alpha_1', \beta_1').   
\end{equation*}
These two points are distinct since $\tau$ is an automorphism of $X$. Using the commutativity of the upper triangle in the diagram, we have
\begin{equation*}
\gamma^{-1} \circ p_1 \circ \phi((\alpha_1, \beta_1)) = p_1(\alpha, \beta) = \alpha, \quad
\gamma^{-1} \circ p_1 \circ \phi((\alpha_1', \beta_1')) = p_1(\alpha, \beta') = \alpha.
\end{equation*}
Similarly, the commutativity of the lower triangle gives
\begin{equation*}
p_2 \circ \phi((\alpha_1, \beta_1)) = \beta_1, \quad p_2 \circ \phi((\alpha_1', \beta_1')) = \beta_1'.
\end{equation*}
Therefore,
\begin{equation*}
 \phi((\alpha_1, \beta_1)) = (\gamma \alpha, \beta_1), \quad \phi((\alpha_1', \beta_1')) = (\gamma \alpha, \beta_1').   
\end{equation*}
Since these points lie on $X'$, which is contained in $F_\C$, the point $\gamma\alpha$ must lie on both lines defined by $\beta_1$ and $\beta_1'$. Moreover, $\beta_1 \ne \beta_1'$, because $\phi$ is an isomorphism and $(\alpha_1, \beta_1) \ne (\alpha_1', \beta_1')$. Hence, $\gamma \alpha$ is the unique point of intersection of the two distinct lines corresponding to $\beta_1$ and $\beta_1'$, and is thus uniquely determined by $\tau$. Since this is true for all $\alpha\in\P^2_\C\backslash C_1$, this implies that the projective transformation $\gamma=\gamma_\tau$ itself is uniquely determined by $\tau$. Therefore, we are left with the commutative diagram

\[
\begin{tikzcd}
& \P^2 \\
X \arrow[ur, "p_1"] \arrow[rr, shift left=2, "\phi \circ \tau"] \arrow[rr, shift right=2, swap, "\phi"] \arrow[dr, swap, "p_2"] 
&& X' \arrow[ul, swap, "\gamma_\tau^{-1} \circ p_1"] \arrow[dl, "p_2"] \\
& \P^2
\end{tikzcd}
\]
and hence $X'$ in $\P^2_\C \times \P^2_\C$, given by
\begin{equation*}
(\gamma_\tau \circ p_1 \circ \tau^{-1},\ p_2)(X),
\end{equation*}
is determined by $\tau$, as necessary. 
$\Box$

\subsection{Proof of Theorem \ref{thm22}:} We are ready to prove the final main theorem.
Recall that $K$ is a number field, and $S$ is a finite set of places of $K$ including all the archimedean places. Let $f\in V_{2,2}(\O_{K,S})$ be an element corresponding to the generic K3 surface $X_{f}$, having good reduction outside $S$. Recall that $X_{f}$ embeds into the flag variety $F_K$, and that this yields two double covers $p_1$ and $p_2$ from $X_f$ to $\P^2_K$. Denote the branch curves of these two maps by $C_{1,f}$ and $C_{2,f}$. By Proposition \ref{prop_branch_covariants}, $C_{1,f}$ and $C_{2,f}$ are cut out by the covariants $I_x(f)$ and $I_z(f)$, respectively.

Let $\frak{p}\not\in S$ be a prime of $K$. We let $\F_{\frak{p}}$ denote $\O_K/\frak{p}$, and let $\bar{f}\in V_{2,2}(\F_{\frak{p}})$ denote the reduction of $f$ modulo $\frak{p}$. By assumption, $X_{\bar{f}}$ in $\P^2_{\F_\pp}$ is a generic K3 surface. Now $I_x$ and $I_z$ have integer coefficients by Proposition \ref{prop_const_covariants}. Hence, the two (smooth) branch curves corresponding to $X_{\bar{f}}$ are cut out by the reduction mod $\pp$ of $I_x(f)$ and $I_z(f)$. Therefore, $C_{1,f}$ and $C_{2,f}$ have good reduction outside $f$, and so the $\GL_3(\cO_{K,S})$-orbits of $I_x(f)$ and $I_z(f)$ are determined (by Theorem \ref{thm2}) upto finitely many choices.

Applying Theorem \ref{thm:FFCC_fin}, we see that the $\GL_3(\C)$-equivalence class $X_f$, and therefore the $\GL_3(\C)$-equivalence class of $f$, has only finitely many possibilities. Theorem \ref{thm22} now follows from an argument identical to the proof of Theorem \ref{thm2}. $\Box$

\bibliographystyle{plain}  
\bibliography{Reference}

\end{document}